\documentclass[12pt, reqno]{amsart}
\usepackage{amsmath, amsthm, amscd, amsfonts, amssymb, graphicx, color}

\newtheorem{df}{Definition}[section]
\newtheorem{thm}[df]{Theorem}

\newtheorem{lem}[df] {Lemma}

\begin{document}
\setcounter{page}{1}

\title[Nilpotent Lie algebras ]{Joint Spectra and Nilpotent Lie Algebras of Linear Transformations }
\author{Enrico Boasso}

\begin{abstract}Given a complex nilpotent finite dimensional Lie algebra of linear 
transformations $L$, in a complex finite dimensional vector space $E$, 
 we study the joint spectra $Sp(L,E)$, $\sigma_{\delta,k}(L,E)$ and 
$\sigma_{\pi,k}(L,E)$. We compute them and we prove that they all coincide
with the set of weights of $L$ for $E$. We also give a new interpretation
of some basic module operations of the Lie algebra $L$ in terms of the joint spectra. \end{abstract}
\maketitle
\section{ Introduction}    
\indent The Taylor joint spectrum is one of the most dynamic and powerful subjects of
multiparameter spectral theory. This spectrum, which was introduced
in 1970 by J. L. Taylor in [9], associates to an n-tuple $a= (a_1,\ldots ,a_n)$
of mutually commuting bounded linear operators on a Banach space $E$, i. e.,
$a_i\in {\mathcal L}(E)$, the algebra of all bounded linear operators on $E$, and
$[a_i,a_j]=0$, $\forall i,j$, $1\le i,j\le n$, a compact non empty subset of 
$\Bbb C^n$, which we denote by $Sp(a,E)$. Besides, 
when $n=1$, the Taylor joint spectrum reduces to the classical spectrum of a 
single operator and, if $\tilde a =(a_i,\ldots ,a_k)$ is the k-tuple, 
$1\le k\le n-1$, of the first k operators of $a$, then the Taylor joint spectrum
satisfies the so called projection property: $Sp(\tilde a ,E)= \pi Sp(a,E)$, 
where $\pi $  is the projection map of $\Bbb C^n$ onto $\Bbb C^k$.\par

\indent In addition, if $E$ is a complex finite dimensional vector space, M. Ch\`o and M.
Takaguchi computed in [4] the Taylor joint spectum of an n-tuple $a$ of commuting
 operators on $E$ and they obtain that $Sp(a,E) = \sigma_{pt}(a)$, the joint point spectrum of $a$, where 
$\sigma_{pt}(a) = \{ \lambda\in\Bbb C^n\colon \exists\, x \in E, x \ne 0\colon a_j(x)=\lambda_j x,
1\le j\le n\}$. Besides in [6], A. McIntosh, A. Pryde and W. Ricker, as a
consequence of a more general result which concerns infinite dimensional spaces too, also 
computed the Taylor joint spectrum for an n-tuple of linear transformations on
a complex finite dimensional vector space, and they compared it with other 
joint spectra.\par

\indent In [2] we defined a joint spectrum for complex solvable finite 
dimensional Lie algebras of operators $L$, acting on a Banach space $E$, and 
we denoted it by $Sp(L,E)$. We also proved that $Sp(L,E)$ is a compact non empty
subset of $L^*$, and that the projection property for ideals still holds. 
Besides, when $L$ is a commutative algebra, $Sp(L,E)$ reduces to the Taylor 
joint spectrum in the following sense. If $\dim\, L = n$, $\{ a_i\}_{(1\le i\le n)}$
is a basis of $L$ and we consider the n-tuple $a = (a_1,\ldots ,a_n)$, then 
$\{(f(a_1),\ldots ,f(a_n))\colon f\in Sp(L,E)\} = Sp(a,E)$, i. e., $Sp(L,E)$, in terms of the
basis of $L^*$ dual of $\{a_i\}_{(1\le i\le n)}$, coincides with the Taylor
joint spectrum of the n-tuple $a$. Then, the following question arises naturally. If $L$  is a complex Lie 
algebra of linear transformations in a complex finite dimensional vector space, what can be said about its joint spectrum? In this article we compute,
under the above assumptions, the joint spectrum of a nilpotent Lie algebra,
moreover, we extend to the nilpotent case the characterizations of [4] and [6].
In addition, our result has a beautiful feature: $Sp(L,E)$ is the set 
of all weights of the Lie algebra $L$ for the vector space $E$. \par

\indent However, we consider a more general situation. In [1] we introduced
a family of joint spectra for complex solvable finite dimensional Lie algebras
of operators acting on a Banach space, which, when the algebra is a commutative 
one, reduce to the Slodkowski joint spectra in the same sense that we have 
explained for the Taylor joint spectrum, see [8]. We also proved the most 
important spectral properties for these joint spectra: compactness, non emptiness 
and the projection property. In this article we also compute these joint spectra for
complex nilpotent Lie algebras of linear transformations in a complex finite 
dimensional vector spaces. Moreover, we also prove that
all the joint spectra considered in this article coincide with the set of 
 weights of the Lie algebra $L$ for the vector space $E$.\par

\indent  We observe that, though all joint spectra introduced in [1] and [2] were
defined for complex solvable Lie algebras, we restrict ourselves to the nilpotent
case for, in the solvable non nilpotent case our result fails.\par
\indent The paper is organized as follows. In Section 2 we review several 
definitions and results of [1] and [2]. We also review several facts related
to the theory of complex nilpotent Lie algebras of linear transformations
in complex finite dimensional vector spaces. In Section 3 we prove our main
theorems and, with them, we give a proof of the compactness and non emptiness
of the joint spectra, for the case under consideration, different from the one of [1] and [2]. 
We also show an example of a solvable non nilpotent Lie algebra
of linear transformations where our result fails. In Section 4 we 
consider some basic module operations of the Lie algebra $L$ in order to compute several joint 
spectra.\par 
                                                 
\section{ Preliminaries}

\indent We briefly recall several definitions and results related to the 
spectrum of a Lie algebra, see [2]. Indeed, in [2] we considered complex 
solvable finite dimensional Lie algebras of linear bounded 
operators acting on a Banach space, however, for our purpose, in this article
we restrict ourselves to the case of complex finite dimensional nilpotent Lie 
algebras of linear transformations defined on finite dimensional vector 
spaces.\par
\indent From now on, $E$ denotes a complex finite dimensional vector space,
${\mathcal L}(E)$ the algebra of all linear transformations defined on $E$ and 
$L$ a complex nilpotent finite dimensional Lie subalgebra of 
${{\mathcal L}(E)}^{op}$, i. e., the algebra ${{\mathcal L}(E)}$ with its opposite product. 
Such an algebra is called a nilpotent Lie algebra of 
linear transformations in the vector space $E$. If $ \dim L = n$ and $f$ is a 
character of $L$, i. e., $f\in L^*$ and $f(L^2) = 0$, where $L^2 = \{ [x,y]\colon x,
 y \in L\}$, then we consider the following  chain complex, 
$(E\otimes\wedge L, d(f))$, where $\wedge L$ denotes
the exterior algebra of $L$, and $d_p(f)$ is as follows,
$$
d_p (f)\colon E\otimes\wedge^p L\rightarrow E\otimes\wedge^{p-1} L,   
$$
\begin{align*} d_p (f) e\langle x_1\wedge\dots\wedge x_p&\rangle  = \sum_{k=1}^p (-1)^{k+1}e(x_k-f(x_k))\langle x_1\wedge\ldots\wedge\hat{x_k}\wedge\dots\wedge x_p\rangle\\
                                                     & + \sum_{1\le k< l\le p} (-1)^{k+l}e\langle [x_k, x_l]\wedge x_1\ldots\wedge\hat{x_k}\wedge\ldots\wedge \hat{x_l}\wedge\ldots \wedge x_p\rangle ,\\ \end{align*} 

\noindent where $\hat{ }$ means deletion. If $p\le 0$ or $p\ge {n+1}$, 
we  define $d_p (f) =0$.\par
 
\indent  If we denote by $H_*((E\otimes\wedge L,d(f)))$ the homology of the complex
$(E\otimes\wedge L, d(f))$, we may state our first definition.\par
\begin{df} With $E$, $L$ and $f$ as above, the
set $\{f\in L^*\colon  f(L^2) =0, H_*((E\otimes\wedge L,d(f)))\neq 0\}$ is the joint spectrum
of $L$ acting on $E$ and it is denoted by $Sp(L,E)$.\end{df}
\indent As we have said, in [2] we proved that $Sp(L,E)$ is a compact non 
empty subset of $L^*$, which reduces to the Taylor joint spectrum when $L$ is a 
commutative algebra, in the sense explained in the Introduction. Besides, if $I$ is an ideal of $L$ and $\pi$ denotes 
the projection map from $L^*$ to $I^*$, then,
$$
Sp(I,E) = \pi (Sp(L,E)),
$$
i. e., the projection property for ideals still holds. With regard to this
property, we ought to mention the paper of C. Ott [7], who pointed out
a gap in the proof of this result and gave another proof of it. In any case,
the projection property remains true.\par
\indent We now give the definition of the Slodkowski joint spectra for Lie algebras. As
 in [1], we give an homological version. However, as we deal with linear 
transformations defined on finite dimensional vector spaces, we slightly 
modify our definition in order to adapt it to our case, for a complete 
exposition of the subject see [8] and [1].\par
\indent If $L$ and $E$ are as above, let us consider the set
$$
\Sigma_p(L,E) = \{ f\in L^*\colon f(L^2)=0, H_p((E\otimes\wedge L,d(f)))\ne 0\},
$$ 
where  $0\le p\le n$, and $\dim L= n$. We now state our definition of the S\l odkowski 
joint spectra for Lie algebras.\par
\begin{df} With $L$ and $E$ as above, $\sigma_{\delta ,k}$
and $\sigma_{\pi ,k}$ are the following joint spectra,
$$
\sigma_{\delta ,k}(L,E)= \bigcup_{0\le p\le k}\Sigma_p(L,E) ,
\hskip2cm
\sigma_{\pi ,k}(L,E)= \bigcup_{k\le p\le n}\Sigma_p(L,E),
$$
where $0\le k\le n$, and $n= \dim L$.\end{df}
\indent We observe that  $\sigma_{\delta ,n}(L,E)= \sigma_{\pi ,0}(L,E)= Sp(L,E)$. Besides, in [1] we
showed that $ \sigma_{\delta ,k}(L,E)$ and $\sigma_{\pi ,k}(L,E)$, $0\le k\le n$,
are compact non empty subsets of $L^*$, which reduce to the S\l odkowski joint spectra when 
$L$ is a commutative algebra. In addition, they also satisfy the projection 
property for ideals, i. e., if $I$ is an ideal of $L$ and $\pi$ denotes the
projection  map from $L^*$ to $I^*$, then,
$$
\sigma_{\delta ,k}(I,E)= \pi (\sigma_{\delta ,k}(L,E)),
\hskip2cm
\sigma_{\pi ,k}(I,E)= \pi (\sigma_{\pi ,k}(L,E)),
$$
where $k$ and $n$ are as above.\par
\indent We shall have occasion to use the theory of weight spaces, however,
 as we deal with complex nilpotent Lie algebras of linear transformations
in complex finite dimensional vector spaces, we restrict our revision to the most important 
results of the theory, essentially Theorems 7 and 12, Chapter II, of the book 
by  N. Jacobson, `Lie Algebras'. For a complete exposition of the subject see
[5, Chapter II].\par
\indent Let $L$ and $E$ be as above. A weight of $L$ for $E$ is a mapping,
$\alpha: x\to \alpha (x)$, of $L$ into $\Bbb C$ such that there exists a non zero
vector $v$ in $E$ with the property, $(x-\alpha (x))^{m_{v,x}}(v)= 0$, for
all $x$ in $L$ and where $m_{v,x}$ belongs to $\Bbb N$. The set of vectors, zero 
included,which satisfy this condition is a subspace of $E$, $E_{\alpha}$, called the
weight space of $E$ corresponding to the weight $\alpha$, 
$$
E_{\alpha}= \{ v\in E\colon \forall x\in L \hbox{ there exists }m_{v,x}\in\Bbb N\hbox{ such that:} (x-\alpha(x))^{m_{v,x}} (v) = 0\}.
$$
\indent As a consequence of our assumptions we have the following properties,
see [5, Chapter II, Theorems 7, 12]:\par
\noindent (i) the weights are linear functions on $L$, which vanish on $L^2$, 
i. e.,  they are characters of $L$,\par
\noindent (ii) $E$ has only a finite number of distinct weights; the weight 
spaces are submodules, and $E$ is the direct sum of them,\par
\noindent (iii) for each weight $\alpha$, the restriction of any $x \in L$ to
$E_{\alpha}$ has only one characteristic root, $\alpha (x)$, with certain
multiplicity, \par
\noindent (iv) there is a basis of $E$ such that for each weight $\alpha$ the
matrices of elements of $L$ in the weight space $E_{\alpha}$ can be taken
simultaneously in the form,
$$
x_{\alpha}=\begin{pmatrix}
                    \alpha (x)&*\\
                    0&\alpha (x)\\
                                       \end{pmatrix}.
$$         
\indent Finally, if $L$ is a complex nilpotent finite dimensional Lie algebra,
by [3, Chapter IV, Section 1], there is a Jordan-H"lder sequence of ideals $(L_i)_{(1\le i\le n)}$
such that,\par
\noindent (i) $\{0\}= L_{0}\subseteq L_i\subseteq L_n= L$,\par
\noindent (ii) $\dim L_i= i$,\par
\noindent (ii) there is a $k$, $0\le k\le n$, such that $L_k= L^2$,\par
\noindent (iv) If $i< j$, $[L_i ,L_j]\subseteq L_{i-1}$.\par
\indent As a consequence, if we consider a basis of $L$, 
$\{ x_j\}_{(1\le j\le n)}$, such that $\{ x_j\}_{(1\le j\le i)}$ is a basis
of $L_i$, we have that
$$
[x_j,x_i]=\sum_{h=1}^{i-1} c^h_{ij}x_h,
$$
where $i<j$.
\section{ The Main Result}
\indent In this section, we compute the joint spectra of a nilpotent Lie
algebra of linear transformations in a complex finite dimensional vector space. Besides, as we have said
in the introduction, and under our assumptions, we give an elementary proof of the 
compactness and the non emptiness of these joint spectra. We also consider an
example which show that, in the solvable non nilpotent case, our result fails.
Let us begin with our characterization of the joint spectra.\par
\begin{thm} Let $E$ be a complex finite dimensional vector space
and $L$ a complex nilpotent Lie subalgebra of ${\mathcal L}(E)^{op}$. Then, if
$\alpha $ is a weight of $L$ for $E$, $\alpha$ belongs to $Sp(L,E)$.\end{thm}
\begin{proof}
\indent As $E$ is the direct sum of its weight spaces, which are $L$-modules 
of $E$, by the definition of the joint spectrum and elementary homological
algebra, we may assume  that there is only one weight of $L$ for $E$. Let us denote it by $\alpha$.\par
\indent We now observe that, as $\alpha$ is a character of $L$, we may consider 
the chain complex $(E\otimes\wedge L ,d(\alpha))$. Moreover, if $\{x_i\}_{(1\le i\le n)}$
is the basis of $L$ defined at the end of Section 2, it is easy to verify 
that,
$$
d_n(\alpha)e\langle x_1\wedge\ldots\wedge x_n\rangle= \sum_{k=1}^n (-1)^{k+1} e(x_k-\alpha (x_k))\langle x_1\wedge\ldots\wedge\hat{x_k}\wedge\ldots\wedge x_n\rangle.
$$
\indent However, as by the properties of the weights reviewed in Section 2,
there exists an ordered basis of $E$, $\{e_i\}_{(1\le i\le m)}$, where $m= dim (E)$,
such that, 
$$
x_k-\alpha(x_k)= \begin{pmatrix}
                            0&*\\
                            0&0\\
                                      \end{pmatrix} ,
$$
for all $k$, $1\le k\le n$. Thus, we have that,
$$        
e_1(x_k-\alpha (x_k))= 0.
$$
\noindent In particular,
$$
d_n(\alpha)e_1\langle x_1\wedge\ldots\wedge x_n\rangle= 0,
$$
which implies that $\alpha$ belongs to $Sp(L,E)$.
\end{proof}
\indent We now state one of our main results.\par
\begin{thm} Let $E$ be a complex finite dimensional vector space and 
$L$ a complex nilpotent Lie subalgbra of ${\mathcal L} (E)^{op}$. Then, 
$$
Sp(L,E)= \{ \alpha\in L^*\colon \alpha \hbox{ is a weight of $L$ for $E$}\}.
$$
\end{thm}
\begin{proof}
\indent By Theorem 3.1, it is enough to see that,
$$
Sp(L,E)\subseteq\{\alpha\in L^*\colon \alpha \hbox{ is a weight of $L$ for $E$}\}.
$$
As in Theorem 3.1 we may suppose that there is only one weight  of $L$ 
for $E$, which we still denote by $\alpha$.\par
\indent In order to conclude the proof, we consider $\beta \in L^*$, $\beta (L^2)= 0$,
such that $\beta$ is not a weight of $L$ for $E$, i. e., in our case
$\alpha\ne\beta$, and by refining an argument of [2], we prove that $\beta$ 
does not belong to $Sp(L,E)$.\par
\indent Let us consider the chain complex associated to $\beta$, 
$(E\otimes\wedge L, d(\beta))$,  the sequence of ideals $(L_i)_{(1\le i\le n)}$,
 and the basis $\{x_i\}_{(1\le i\le n)}$ of $L$ reviewed in Section 2. 
As $\alpha\ne \beta$, there is a $j$, $k+1\le j\le n$, where $ k= \dim L^2$ and $n= \dim L$,
such that $\alpha (x_j)\ne\beta (x_j)$. Now, if $L'$ is the ideal of codimension 1 
of $L$ generated by $\{ x_i\}_{(1\le i\le n, i\ne j)}$, we have that $L'\oplus\langle x_j\rangle= L$.
Moreover, we may deconpose  $E\otimes\wedge^p L$ in the following way,
$$
E\otimes\wedge^p L= (E\otimes\wedge^p L')\oplus ((E\otimes\wedge^{p-1} L')\wedge\langle x_j\rangle).
$$
If $\tilde\beta$ denotes the restriction of $\beta$  to $L'$, then we may 
consider the complex $(E\otimes\wedge L',d(\tilde\beta))$ and, as $L'$
is an ideal of codimension 1 of $L$, we may decompose $d_p(\beta)$ as follows,
$$
d_p(\beta):E\otimes\wedge^p L'\to E\otimes\wedge^{p-1} L',
$$
$$
d_p(\beta)= d_p(\tilde\beta),
$$
$$    
d_p(\beta): E\otimes\wedge^{p-1} L'\wedge\langle x_j\rangle\to (E\otimes\wedge^{p-1} L')\oplus (E\otimes\wedge^{p-2} L'\wedge\langle x_j\rangle),
$$
$$
d_p(\beta)(a\wedge\langle x_j \rangle)= (-1)^p L_{p-1}(a) +(d_{p-1}(\tilde\beta)(a))\wedge\langle x_j\rangle,
$$
where $a\in E\otimes\wedge^{p-1} L'$ and $L_{p-1}$ is the linear endomorphism 
defined on $E\otimes\wedge^{p-1} L'$ by

\begin{align*}
L_{p-1} e\langle y_1\wedge\ldots\wedge y_{p-1}\rangle=& e(x_j-\beta (x_j))\langle y_1\wedge\ldots\wedge y_{p-1}\rangle\\
                                                      & +\sum_{1\le l\le p-1}(-1)^l e\langle[x_j,y_l]\wedge y_1\wedge\ldots\wedge\hat{ y_l}\wedge\ldots\wedge y_{p-1}\rangle,\\\end{align*}

\noindent where $\hat{ }$ means deletion and $\{y_h\}_{(1\le h\le p-1)}$ belong to $L'$. \par
\indent For $s\in [\![1,j-1]\!]$, $[x_j,x_s]= \sum_{h=1}^{s-1} c^h_{s j} x_h$, and for $s\in [\![j+1, n]\!]$,
$[x_j, x_s]= \sum_{h=1}^{j-1}(-c^h_{j s})x_h$. Thus, by the properties of weights considered in Section 2, 
it is easy to see that for each $ p$ there is a basis of 
$E\otimes\wedge^p L'$ such that in this basis $L_p$ is an upper triangular 
matrix with diagonal entries $\alpha (x_j)-\beta (x_j)$. However, as
$\alpha (x_j)\ne\beta (x_j)$, $L_p$ is an invertible map for each $p$. Thus,
as in [2], we may construct an homotopy operator for the complex 
$(E\otimes\wedge L, d(\beta))$. However, this fact implies that 
$H_* ((E\otimes\wedge L,d(\beta))$= 0, or equivalently, $\beta$ does not belong to
$Sp(L,E)$
\end{proof}
\indent As a consequence of Theorem 3.2, we have, under our assumptions, another
proof of the compactness  and  non emptiness of the spectrum.\par
\begin{thm} Let $E$ be a complex finite dimensional vector space and
 $L$ a complex nilpotent Lie subalgebra of ${\mathcal L}(E)^{op}$. Then, $Sp(L,E)$
is a compact non empty subset of $L^*$.\end{thm}
\indent We now consider the joint spectra $\sigma_{\delta k}(L,E)$ and 
$\sigma_{\pi k}(L,E)$, $ 0\le k\le n$. We first state a lemma which we need for our
characterization of these joint spectra.\par
\begin{lem} Let $E$ be a complex finite dimensional vector space and $L$
a complex nilpotent Lie subalgebra of ${\mathcal L} (E)^{op}$ such that $\dim L= n$. 
Then,\par
\noindent \rm{(i)} $\Sigma_n(L,E)= Sp(L,E)$,\par
 \noindent \rm{(ii)} $\Sigma_0(L,E)= Sp(L,E)$.\end{lem}
\begin{proof}
\indent If one looks at Theorem 3.1, a consequence of its proof is the fact that,
 if $\alpha$ is a weight of $L$ for $E$, then $\alpha\in \sum_n(L,E)$. However, by
Theorem 3.2 we have:
$$
Sp(L,E) \subseteq\Sigma_n(L,E)\subseteq Sp(L,E).
$$
\indent In order to prove (ii), we consider the chain complex involved in the 
definition of $Sp(L,E)$, $(E\otimes\wedge L,d(f))$, where $f$ is a character of $L$, and we study $d_1(f)$. We shall see that if $\alpha$ is a weight
of $L$ for $E$, $d_1(\alpha)$ is not onto, equivalently, 
$ 0\ne E/R(d_1(\alpha))= H_0(E\otimes\wedge L,d(\alpha))$, i. e., $\alpha$
belongs to $\sum_0(L,E)$. Then, by Theorem 3.2, as in the proof of (i), 
we conclude (ii).\par
\indent Let us prove the previous claim. We consider $\alpha$ a weight 
of $L$ for $E$. Besides, by elementary homological arguments  and the properties 
 of weights, we may suppose that $E= E_{\alpha}$. If this is the case, let us consider 
the chain complex $(E\otimes\wedge L,d(\alpha))$ and the map $d_1(\alpha)$,
$$
d_1(\alpha):E\otimes\wedge^1 L\to E,
$$
$$
d_1(\alpha)(e\langle x\rangle)= e(x-\alpha(x)).
$$
\indent However, by the revision of Section 2, there is a basis of $E$, $\{e_i\}_{(1\le i\le m)}$, where $m= \dim (E)$,
such that the matrix of any $x-\alpha (x)$, in this basis, has the form,
$$
x-\alpha (x)=\begin{pmatrix}
                0&*\\      
                0&0\\
                              \end{pmatrix}.
$$
\indent Then, $e_m$ does not belong to $R(d_0(\alpha))$, which implies our claim.\end{proof}
\indent We now may study the relation among $Sp(L,E)$, $\sigma_{\delta k}(L,E)$
and $\sigma_{\pi k}(L,E)$, $0\le k\le n$, for nilpotent Lie algebras of
linear transformations in complex finite dimensional vector spaces. The following
theorem states the relation among then.\par
\begin{thm} Let $E$ be a complex finite dimensional vector space and 
$L$ a complex nilpotent Lie algebra of ${\mathcal L}(E)^{op}$. Then, we have the
following identity,
$$
Sp(L,E)= \sigma_{\delta k}(L,E)= \sigma_{\pi k}(L,E),
$$
where $0\le k\le n$, and $dim (L)= n$.\end{thm}
\begin{proof}
\indent The proof is a consequence of the following observation. By Lemma 3.4 and Definitions 2.1 and 2.2, we have that,
$$
Sp(L,E)= \Sigma_0(L,E)\subseteq\sigma_{\delta k}(L,E)\subseteq Sp(L,E) ,
$$
$$
Sp(L,E)= \Sigma_n(L,E)\subseteq\sigma_{\pi k}(L,E)\subseteq Sp(L,E),
$$
where $ 0\le k\le n$, and $\dim (L)= n$.  
\end{proof}
\indent As a consequence of Theorems 3.3 and 3.5, we have the following result.
\begin{thm} Let $E$ be a complex finite dimensional vector space and $L$
a complex nilpotent Lie subalgebra of ${\mathcal L}(E)^{op}$. Then 
$\sigma_{\delta k}(L,E)$ and $\sigma_{\pi k}(L,E)$ are compact non empty subsets
of $L^*$, where $0\le k\le n$, and $ n= \dim (L).$\end{thm} 
\indent We now give an example in order to show that, for the solvable non nilpotent
case, our characterization of the joint spectra fails.\par
\indent Let us consider the solvable Lie algebra on two generators, $L=
\langle y\rangle\oplus\langle x\rangle$, with the bracket $[x,y]^{op}= y$, i. e., $[x,y]= -y$. $L$ may
be viewed as a subalgebra of ${\mathcal L}(\Bbb C^2)^{op}$ as follows,
$$
y=\begin{pmatrix}                                                                                   
            1&1 \\                                                                                     
           -1&-1\\
                  \end{pmatrix},
\hskip2cm
x=\begin{pmatrix}
              0&\frac{1}{2}\\
             \frac{1}{2}&0\\
                               \end{pmatrix}.
$$                
\indent An easy calculation shows that the weights of $L$ for $\Bbb C^2$
are, in terms of the dual basis of $y$ and $x$, $(0,\frac{1}{2})$ and
$(0,\frac{-1}{2})$. However, if we still consider the previous basis,
$Sp(L,\Bbb C^2)= \{(0,\frac{1}{2}), (0,\frac{-3}{2})\}$. 

\section{Some Consequences}
\indent In this section we shall see some consequences of the main theorems. We 
begin with a result which relates the weights of differents ideals.\par
\begin{thm} Let $E$ be a complex finite dimensional vector space, 
$L$ a complex nilpotent Lie subalgebra of ${\mathcal L} (E)^{op}$ and $I$ an ideal of 
$L$. Then, if $\alpha$ is a weight of $L$ for $E$, its restriction to $I$,
$\alpha\mid I$, is a weight of $I$ for $E$. Reciprocally, if $\beta$ is a
 weight of $I$ for $E$, there exists $\alpha$, a weight of $L$ for $E$, such 
that $\beta= \alpha\mid I$. In particular, if $I$ is contained in $L^2$, 
$I$ has only one weight for $E$: $\beta= 0$.\end{thm}
\begin{proof}
\indent The proof is a consequence of Theorem 3.2  and Theorem 3
of [2], the projection property for ideals of the spectrum.
\end{proof}
\indent From now to the end of this section, we deal with two basic module
operations of the Lie algebra $L$, the contragradient module and the tensor product, see [5, Chapter III]. We shall
compute several new joint spectra by means of these module operations. Moreover, these results give a new interpretation
 of these module operations in terms of the joint spectra. More precisely, we consider
Propositions 3 and 4 of [5, Chapter III], and we state then as properties of the joint spectra. Let us begin with 
the contragradient module.\par
\indent If $L$ and $E$ are as usual, we consider the space of linear 
functionals on $E$, $E^*$. If $x$ belongs to $L$ and  $x^*$ denotes its 
adjoint, then it is easy to see that the set $L'$ 
$$
L'= \{x^*\colon x\in L\},
$$
with the natural bracket, defines a complex nilpotent  Lie subalgebra of
${\mathcal L}(E^*)^{op} $, which is isomorphic to $L$ with the opposite bracket.
We now compute the joint spectra of $L'$ on $E^*$.\par
\begin{thm} Let $E$ be a complex finite dimensional vector space
 and $L$ a complex nilpotent Lie subalgebra of ${\mathcal L}(E)^{op}$. Let $E^*$
be the dual vector space of $E$ and $L'$  the nilpotent Lie subalgbra of 
${\mathcal L}(E^*)^{op}$ defined above. Then,
$$
Sp(L',E^*)= Sp(L,E), 
$$
$$
\sigma_{\delta k}(L',E^*)= \sigma_{\delta k}(L,E),
$$
$$
\sigma_{\pi k} (L',E^*)= \sigma_{\pi k }(L,E),
$$
where $0\le k\le n$, and $n=\dim (L)= \dim (L')$.\end{thm}
\begin{proof}
\indent As we may consider the inclusion as a representation of $L$ in 
${\mathcal L}(E)^{op}$, the proof is a consequence of a slight modification of
 the proof of [5, Chapter III, Proposition 3] and Theorems 3.2 and 3.5.
\end{proof}
\indent Finally, we study the joint spectrum of a tensor product. If $E_1$ and
$E_2$ are two complex finite dimensional vector spaces and $L_1$ and $L_2$ are
two complex nilpotent Lie subalgebras of ${\mathcal L} (E_1)^{op}$ and 
${\mathcal L}(E_2)^{op}$, respectively, then we consider the tensor product of $E_1$
and $E_2$, $E_1\otimes E_2$, and the nilpotent Lie subalgebras of 
${\mathcal L}(E_1\otimes E_2)^{op}$, $\tilde L_1$ and $\tilde L_2$, defined by,
$$
\tilde L_1= \{x\otimes 1\colon x\in L_1\},  
\hskip2cm
\tilde L_2= \{1\otimes y\colon y\in L_2\},
$$
where 1 denotes the identity of the corresponding space. As in 
${\mathcal L}(E_1\otimes E_2)$,  $[\tilde L_1 ,\tilde L_2]= 0$, the set,
$$
\{ x\otimes 1+ 1\otimes y\colon x\in L_1, y\in L_2\},
$$
may be viewed  as the direct product of the Lie algebras $L_1$ and $L_2$,
$L_1\times L_2$, with the natural bracket of ${\mathcal L}(E_1\otimes E_2)^{op}$.
Besides, $\tilde L_1$ and $\tilde L_2$ are two ideals of the complex 
nilpotent Lie algebra $L_1\times L_2$. Our objective is to compute the joint 
spectra of $L_1\times L_2$ on $E_1\otimes E_2$. The following theorem 
describes the situation.\par
\begin{thm} Let $E_1$ and $E_2$ be two complex finite dimensional 
vector spaces and $L_1$ and $L_2$ two complex nilpotent  subalgebras of
${\mathcal L}(E_1)^{op}$ and ${\mathcal L}(E_2)^{op}$, respectively. Then,
$$
Sp(L_1\times L_2, E_1\otimes E_2)= Sp(L_1, E_1)\times Sp(L_2, E_2),
$$
$$
\sigma_{\delta k}(L_1\times L_2, E_1\otimes E_2)= Sp(L_1, E_1)\times Sp(L_2, E_2),
$$
$$
\sigma_{\pi k}(L_1\times L_2, E_1\otimes E_2)= Sp(L_1, E_1)\times Sp(L_2, E_2),
$$
where $ 0\le k\le n+m$, $n= \dim(L_1)$ and $m= \dim (L_2)$.\end{thm}
\begin{proof}
\indent Let us consider the decomposition of $E_1$, respectively $E_2$, as direct
sum of its weight $L_1$-submodules, respectively its weight $L_2$-submodules,
$$
E_1= \oplus_{\alpha\in \Phi} E_{1\alpha},
\hskip2cm
E_2= \oplus_{\beta\in\Psi} E_{2 \beta},
$$
where $\Phi= \{ \alpha\in L_1^*\colon \alpha \hbox{ is weight of $L_1$ for $E_1$} \}= Sp(L_1,E_1)$ 
and, $\Psi= \{ \beta\in L_2^*\colon \beta \hbox{ is weight of $L_2$ for $E_2$} \}= Sp(L_2,E_2)$.
Then,
$$
E_1\otimes E_2 =  \oplus_{(\alpha, \beta)\in \Phi\times\Psi} E_{1 \alpha}\otimes E_{2 \beta}.
$$
\indent An easy calculation shows that each $E_{1 \alpha}\otimes E_{2 \beta}$,
$(\alpha, \beta)\in \Phi\times\Psi$, is a $L_1\times L_2$-submodule of $E_1\otimes E_2$.
Besides, as $ E_{1 \alpha}\ne 0$ and $E_{2 \beta}\ne 0$, 
$E_{1 \alpha}\otimes E_{2 \beta}\ne 0$. In addition, by the definition of $\Phi$ 
and $\Psi$, if $(\alpha_1, \beta_1)\ne (\alpha_2,\beta_2)$, there exists an 
element of $L_1\times L_2$, $(x,y)$, such that $(\alpha_1 , \beta_1)(x,y)\ne (\alpha_2 ,\beta_2)(x,y)$.
Finally, in order to conclude the proof, by Theorem 3.5 and [5, Chapter II, Theorem 7], it is 
enough to prove that $E_{1 \alpha}\otimes E_{2 \beta}$ is the weight module
of $L_1\times L_2$ for the weight $(\alpha ,\beta)$. However, this fact may be verify by a 
slight modification of [5, Chapter III, Proposition 4].
\end{proof}

\bibliographystyle{amsplain}

\vskip.5cm

Enrico Boasso\par
E-mail address: enrico\_odisseo@yahoo.it

\end{document}